\documentclass[journal,twoside,web]{ieeecolor}

\usepackage{amsmath,amssymb,amsfonts}
\usepackage{algorithmic}
\usepackage{graphicx}
\usepackage{textcomp}
\usepackage{graphicx}
\usepackage{subfigure}
\usepackage{epsfig} 
\usepackage{cite}
\usepackage{lcsys}

\usepackage{graphics} 
\usepackage{epsfig} 
\usepackage{amsmath} 
\usepackage{amssymb}  
\usepackage{algorithm}
\usepackage{algorithmic}
\usepackage{verbatim}
\allowdisplaybreaks[4]
\usepackage{tikz}
\usepackage{booktabs}
\usepackage{multirow}

\newtheorem{proposition}{Proposition}
\newtheorem{theorem}{Theorem}
\newtheorem{lemma}{Lemma}
\newenvironment{remark}
  {\par\smallskip\textit{Remark.}\ \normalfont}
  {\par\smallskip}

\usepackage{mathtools}
\usepackage{comment}
\usepackage[table,dvipsnames]{xcolor}
\definecolor{subsectioncolor}{rgb}{0.067,0.627,0.859}
\definecolor{nblue}{rgb}{0,0.263,0.576}
\definecolor{mblue}{rgb}{0.075,0.541,0.855}
\usepackage{soul}

\usepackage{hyperref}
\hypersetup{
    colorlinks=true,
    linkcolor=blue,
    citecolor=blue,
    urlcolor=blue
}

\DeclareSymbolFont{bbold}{U}{bbold}{m}{n}
\DeclareSymbolFontAlphabet{\mathbbold}{bbold}

\usepackage{accents}

\DeclareMathOperator{\img}{im}
\DeclareMathOperator{\spn}{span}
\DeclareMathOperator{\inte}{int}

\newcommand{\E}{\mathcal{E}}
\newcommand{\N}{\mathcal{N}}

\usepackage{eso-pic}

\newcommand{\IEEEcopyrightnotice}{%
\AddToShipoutPictureFG*{%
  \AtPageLowerLeft{%
    \raisebox{0.3in}{%
      \hspace*{\dimexpr(\paperwidth-\textwidth)/2\relax}%
      \parbox{\textwidth}{\centering\footnotesize
      \copyright~2026 IEEE. Personal use of this material is permitted.
      Permission from IEEE must be obtained for all other uses, including
      reprinting/republishing this material for advertising or promotional purposes,
      creating new collective works for resale or redistribution to servers or lists,
      or reuse of any copyrighted components of this work in other works.}%
    }%
  }%
}%
}

\begin{document}

\def\BibTeX{{\rm B\kern-.05em{\sc i\kern-.025em b}\kern-.08em
    T\kern-.1667em\lower.7ex\hbox{E}\kern-.125emX}}
\markboth{\journalname, VOL. XX, NO. XX, XXXX 2017}
{Author \MakeLowercase{\textit{et al.}}: Preparation of Papers for IEEE Control Systems Letters (August 2022)}

\title{A Cycle-Based Solvability Condition for Real Power Flow Equations}

\author{
    Puskar Neupane,~\IEEEmembership{Student Member,~IEEE,}
    and Bai Cui,~\IEEEmembership{Member,~IEEE}
\thanks{The authors are with the Department of Electrical and Computer Engineering, Iowa State University, Ames, IA, USA. Emails: {\tt \{puskar, baicui\}@iastate.edu}.} \thanks{This material is based upon work supported by the U.S. National Science Foundation under Grant No. DMS-2523935.}
}

\maketitle
\thispagestyle{empty}
\IEEEcopyrightnotice

\begin{abstract}
Certifying power flow solvability is important for
reliable power system operations under volatile operating conditions, but solving power flow equations repeatedly can be costly and may encounter convergence issues. In this paper, we develop an explicit cycle-based solvability condition for the lossless real power flow equations on meshed networks. We decompose every feasible nodal balance solution into a particular flow plus a cycle flow correction vector. The power flow problem is then reduced to enforcing edge-wise feasibility and cycle consistency. We show that the cycle consistency function is strongly monotone and is the gradient of a strongly convex energy function. By exploiting these properties, we derive an explicit condition for the existence and uniqueness of a power flow solution with bounded angle difference. The resulting condition is invariant under the choice of cycle basis and can be verified through simple algebraic computations. Numerical results on standard test systems show that the proposed condition is significantly less conservative than existing sufficient conditions and closely approximates true loading limits.
\end{abstract}


\section{Introduction}

\IEEEPARstart{T}{he} nonlinear nature of the AC power flow equations makes the computation of power flow solutions challenging. One common approach to obtaining approximate solutions is through iterative numerical methods. Among these, Newton-Raphson (NR) method is the most widely used technique in practice. However, the convergence of the NR method depends strongly on the choice of the initial guess and is not guaranteed in general. Moreover, even when convergence occurs, the method may fail to converge to the physically relevant operating solution. For example, in heavily loaded systems, multiple power flow solutions can coexist and may lie very close to each other, increasing the possibility that NR converges to an unstable solution. These limitations have motivated research into alternative iterative methods with improved convergence and robustness properties {\cite{FPPF_Chen}}. In parallel, analytical approaches have been developed to characterize the existence of power flow solutions in lieu of numerical solution routines~{\cite{energy_function, Park2021,  CUI2025112158, Porco_II}}. They provide fast certificates of solution existence without solving power flow equations and are particularly suitable in scenarios where repeated power flow evaluations are needed. The problem of determining solvability limits has therefore attracted considerable attention in recent decades. 

In the early 1980s, Wu and Kumagai~\cite{Wu1982} proposed a sufficient condition for the existence of power flow solution. Later in the decade, Thorp \textit{et al.}~\cite{THORP198666} investigated the existence of solutions to the reactive power balance problem and proposed conditions for assessing solvability following system disturbances. Subsequently, an analytical result on the existence and uniqueness of the power flow solution for radial networks was proposed~\cite{Chiang1990}. Other early works along this line of research include~\cite{Marija_1992, Miu2000}.

The derivation of sufficient conditions has also been studied in different ways in the literature. The paper in~\cite{energy_function} used an energy function as a tool in order to propose necessary and sufficient conditions for the existence of a power flow solution. The topology of the power network has also been leveraged in prior work to derive solvability conditions for the AC power flow equations~\cite{Park2021}. In~\cite{Porco_I}, a novel iterative algorithm is proposed for solving power flow equations in both meshed and radial networks, while~\cite{Porco_II} derives a sufficient condition for the existence of a high-voltage solution in radial networks. In a related line of work, power system networks are modeled as systems of coupled oscillators, and solvability conditions are obtained by drawing analogies with synchronization phenomena in Kuramoto oscillator models. In particular, the work in~\cite{Cutset_Projection} studies the lossless power flow problem by expressing the dynamics in an edge-balance fixed-point form via a cutset projection operator and establishes existence and uniqueness of solutions using the Brouwer Fixed Point Theorem.

The work in~\cite{Bolognani} proposed a sufficient condition for the coupled full-power model to verify the existence of power flow solutions and provided an approximate solution to the nonlinear power flow equations. Having sufficient conditions for the solvability of the coupled full power flow equation reduces computational burden~\cite{Cong2018} and offers valuable insights for security monitoring and the design of effective control strategies. These sufficient conditions can be incorporated as constraints in the optimal power flow problem~\cite{2018Cui} and used to ensure voltage stability of the system~\cite{Application, CUI2025112158}.

Existing research has successfully developed sufficient conditions for the existence of power flow solutions. However, these conditions are often conservative and rely on restrictive assumptions. Given the inherently meshed nature of real-world transmission networks, it is important to develop solvability conditions that explicitly exploit such structure. Motivated by this, we leverage the topological structure of meshed networks, particularly their cyclic characteristics, to derive a sharper sufficient condition for the existence of a solution to the lossless real power flow equation. Since solvability conditions for the decoupled power flow model can help pave the way toward sharper and more robust conditions for the full coupled AC power flow equations, we focus in this paper on the lossless real power flow case.

The remainder of the paper is organized as follows. Section~\ref{Background} presents background and preliminaries. Section~\ref{PF_Formulation} introduces the problem formulation and assumptions. Section~\ref{Suff_Condition} derives the sufficient condition for the existence of a power flow solution. Section~\ref{Methodology} describes the methodology and implementation details. Section~\ref{Results} presents the numerical results on standard test cases and comparison with existing work. Finally, Section~\ref{Conclusion} concludes the paper.

\section{Background and Preliminaries}
\label{Background}

\subsection{Power System Model}

A power network is represented by a connected, undirected, weighted graph $G (\N, \E)$ where $\N$ is the set of nodes (buses) and $\E \subseteq \N\times\N$ is the set of edges (lines). An arbitrary orientation is assigned to each edge so that the edge is oriented from its source node to the sink node. We denote the number of buses and the number of lines by $n$ and $m$, respectively. For each line $\{i,j\}\in \E$, let $g_{ij}$ and $b_{ij}$ denote the series conductance and susceptance, respectively, so that the line admittance is given by $y_{ij}=g_{ij}+jb_{ij}$. In this work, we consider a lossless power system without shunt elements; therefore, $g_{ij}=0$ for all $(i,j)\in \E$, and each line admittance is purely imaginary, i.e., $y_{ij}=jb_{ij}$. The bus admittance matrix $Y \in \mathbb{C}^{n\times n}$ captures both the network's topology and its weights. The diagonal element $Y_{ii}$ is equal to $ jB_{ii}$, where $B_{ii}$ is the sum of the susceptances $b_{ij}$ of all lines incident to bus $i$. The off-diagonal element $Y_{ij}$ is $ jB_{ij}$, where $B_{ij}$ equals the negative of the susceptance of the line connecting buses $i$ and $j$. For all $i \in N$, $V_i$ is the voltage magnitude and $\theta_i$ is the phase angle. The full AC power flow model consists of both active and reactive power equations. In this work, we study only the lossless real power flow equations
\begin{equation}
P_i =V_i \sum_{j=1}^{n} V_j B_{ij} \sin(\theta_i - \theta_j), \quad i \in \N,
\label{PF:Real_Power_lossless}
\end{equation}
where the unknown variables are the phase angles $\theta$, while the voltage magnitudes $V$ are fixed and known, corresponding to the base case operating point of the system.

\subsection{Algebraic Graph Theory}
Considering an arbitrary orientation assigned to each edge of the graph $G(\N,\E)$, the node-edge incidence matrix is given by $A \in \{-1,1,0 \}^{n\times m}$. The proposed sufficient condition is independent of the chosen orientation. The element $A_{ke}$ will be $+1$ if $k$ is the source node of the edge $e$, $-1$ if $k$ is the sink node of the edge $e$, and zero otherwise. The diagonal susceptance matrix $D \in \mathbb{R}^{m \times m}$ encodes the edge weights. For an edge $e$ connecting buses $i$ and $j$, the corresponding diagonal entry of $D$ is given by $D_{ee} = V_i V_j B_{ij}$. Throughout the paper, we assume inductive lines for every line $(i,j)\in\mathcal E$, and hence
$D\succ 0$. The Laplacian matrix of the graph $G$ is defined as $L = A D A^\top \in \mathbb{R}^{n\times n}$, which has a one-dimensional nullspace $\ker(L) = \spn\left\{\mathbbold{1}\right\}$, where $\mathbbold{1}$ is the vector of all $1$'s. The (Moore-Penrose) pseudoinverse of $L$ is denoted by $L^\dagger$.

A walk in a graph is a sequence of nodes such that each consecutive pair of nodes is connected by an edge. If the walk starts and ends at the same node, it is called a closed walk. A cycle is a closed walk in which no node is repeated except the first and last. A cycle in a graph can be represented by a cycle vector $c_k \in \mathbb{R}^m$, whose entries indicate whether each edge belongs to the cycle and, if so, whether its orientation agrees with the chosen traversal direction of the cycle. Specifically,
\begin{equation*}
(c_k)_e =
\begin{cases}
+1, & \text {if $e$'s orientation aligns with the traversal},\\
-1, & \text{if $e$'s orientation is opposite to the traversal},\\
0,  & \text{otherwise}.
\end{cases}
\end{equation*}

Any set of linearly independent cycle vectors that spans the cycle space forms a cycle basis. The cycle space of the graph $G$ is $\ker(A)$, and its dimension is $q = m-n+1$ \cite[Theorem $9.5$]{FB-LNS}. The cutset space of the graph $G$ is the column space of $A^\top$. Moreover, the cycle space and cutset space are orthogonal complements in $\mathbb{R}^m$, so that
\begin{equation*}
    \mathbb{R}^m = \ker(A) \oplus \img(A^\top).
\end{equation*}
Given a cycle basis $\Sigma = \{ c_1, \ldots, c_{m-n+1} \}$, $C_{\Sigma}$ denotes the cycle basis matrix whose columns are the basis cycle vectors.

\section{Power Flow equation}
\label{PF_Formulation}

Collecting the bus power injections into the vector $P \in ~\mathbb{R}^n$, \eqref{PF:Real_Power_lossless} can be written compactly as
\begin{equation}
P = A D \sin(A^\top \theta).
\label{PF:matrix_form}
\end{equation}
For a real-value vector $x$, the functions $\sin(x)$ and $\arcsin(x)$ are understood componentwise. In this paper, we are interested in deriving a sufficient condition on the nodal injection $P \in \mathbbold{1}^\perp \coloneqq \{x\in\mathbb{R}^n : \mathbbold{1}^\top x = 0\}$ under which \eqref{PF:matrix_form} admits a solution.

We see from \eqref{PF:Real_Power_lossless} and \eqref{PF:matrix_form} that the nodal power injection $P_i$ is expressed in terms of the angle differences $\theta_i - \theta_j$ for lines $(i,j)$ incident to bus $i$. For each line $(i,j) \in \E$, define $\tilde{b}_{ij}:= V_iV_j B_{ij}$. Then the real power flow on line $(i,j)$ is
\begin{equation}
f_{ij} = \tilde{b}_{ij} \sin(\theta_i - \theta_j).
\label{PF:Real_Power_lossless_branch}
\end{equation}
In vector form, this becomes 
\begin{equation}
    f \coloneqq  D \sin(A^\top \theta) \in \mathbb{R}^m.
\end{equation}
Using this definition, the compact power flow equation can be rewritten as
\begin{equation}
P = Af.
\label{nodal_balance}
\end{equation}
We refer to \eqref{nodal_balance} as the \emph{nodal balance equation}, and any vector $f \in \mathbb{R}^m$ satisfying it is called a \emph{line flow solution}.

The nodal balance equation \eqref{nodal_balance} is solvable if and only if $P \in \img(A)$. Since the graph is connected, $\img(A) = \mathbbold{1}^\perp$, and therefore \eqref{nodal_balance} is solvable for every nodal injection vector \(P \in \mathbbold{1}^\perp\). Moreover, because \(L=ADA^\top\) and \(\img(L)=\mathbbold{1}^\perp\), a particular solution is given by
\begin{equation}
\hat f = DA^\top L^\dagger P.
\end{equation}
Indeed,
\begin{equation}
A\hat f = ADA^\top L^\dagger P = LL^\dagger P = P,
\end{equation}
where the last equality follows from the fact that $LL^\dagger$ is the orthogonal projector onto $\img(L) = \mathbbold{1}^\perp$ and $P \in \mathbbold{1}^\perp$.

A standard result in linear algebra~\cite{Strang} states that, if $\hat{f}$ is a particular solution of~\eqref{nodal_balance}, then the set of all line flow solutions to~\eqref{nodal_balance} is given by
\begin{equation}
f = \hat{f} + s, \qquad s \in \ker(A).
\label{feasible_solution}
\end{equation}
Since $\ker(A)$ is the cycle space of the graph, any $s \in \ker(A)$ is referred to as a cycle flow. It follows from \eqref{feasible_solution} that any two line flow solutions differ by a cycle flow.

While every power flow solution $\theta$ induces a line flow solution through $f_{ij} = \tilde{b}_{ij} \sin(\theta_i - \theta_j)$, the converse is not true in general: a line flow solution $f$ does not necessarily correspond to a power flow solution. The next proposition characterizes exactly when this correspondence holds.

\begin{proposition}
\label{power_flow_solution}
Given a line flow solution $f^*$ satisfying \eqref{nodal_balance} and a cycle basis matrix $C_{\Sigma}$ of $G$, there exists a {unique} power flow solution $\theta^*$ {up to an additive constant} such that
\begin{equation}
f^*_{ij} = \tilde{b}_{ij}\sin(\theta^*_i-\theta^*_j),\quad  |\theta^*_i-\theta^*_j| \le \frac{\pi}{2}, \quad \forall (i,j)\in \E,
\label{main_condition}
\end{equation}
if and only if the following two conditions hold:
\begin{enumerate}
\item Edge-wise feasibility: \label{item:cond1}
\begin{equation}
|f^*_{ij}| \le \tilde{b}_{ij}, \quad \forall (i,j) \in \E, 
\label{edge_wise}
\end{equation}

\item Cycle consistency: \label{item:cond2}
\begin{equation} \label{cycle_consistent}
C_{\Sigma}^\top \arcsin(D^{-1}f^*)=  \mathbbold{0}.
\end{equation}
\end{enumerate}
\end{proposition}

\begin{proof}
We first prove the ``only if'' part.
Given a power flow solution $\theta^*$ for $f^*$ as given by equation~\eqref{main_condition}, we have
\begin{equation}
    f^*_{ij}/\tilde{b}_{ij} = \sin(\theta^*_i-\theta^*_j) \in [-1,1]. 
\label{Condn_i}
\end{equation}
So, condition \ref{item:cond1} is satisfied. Since $|\theta_i^* - \theta_j^*|$ is bounded by $\pi/2$, we know
\begin{equation}
    \arcsin \bigl(f^*_{ij} / \tilde{b}_{ij} \bigr) = \theta^*_i-\theta^*_j.
\end{equation}
The above equation can be written in vector form as:
\begin{equation}
    \arcsin(D^{-1}f^*) = A^\top \theta^*.
\end{equation}
Premultiplying by $C_\Sigma$ gives
\begin{equation}
    C_\Sigma^\top \arcsin(D^{-1}f^*) = C_\Sigma^\top A^\top \theta^* = (AC_\Sigma)^\top \theta^* =  \mathbbold{0},
\end{equation}
where the last equality holds since each column of $C_\Sigma$ lies in $\ker(A)$. Therefore, condition \ref{item:cond2} is also satisfied.

For the other direction, we show that there exists a $\theta^*$ corresponding to $f^*$ whenever conditions \ref{item:cond1} and \ref{item:cond2} hold. By condition \ref{item:cond1}, the vector
\begin{equation}
    \delta \coloneqq \arcsin(D^{-1}f^*)
\end{equation}
is well defined, with each component lying in $[-\pi/2,\pi/2]$. In addition, condition \ref{item:cond2} implies $C_\Sigma^\top\delta =  \mathbbold{0}$. Since the cycle space and cutset space are orthogonal complements to each other, we have $\ker(C_{\Sigma}^\top) = \img(A^\top)$. It follows that there exists $\theta^* \in \mathbb{R}^n$ such that $\delta = A^\top \theta^*$. Substituting the definition of $\delta$ and applying $\sin(\cdot)$ componentwise yields
\begin{equation}
    D^{-1}f^* = \sin(A^\top\theta^*).
\end{equation}
Equivalently, $f_{ij}^* = \tilde{b}_{ij}\sin(\theta^*_i-\theta^*_j)$ for every line $(i,j) \in \E$. Since each $\delta_{ij} \in [-\pi/2,\pi/2]$, the corresponding angle difference satisfies $|\theta_i - \theta_j| \le \pi/2$ for each $(i,j)\in\E$. Thus, $\theta^*$ is a power flow solution satisfying \eqref{main_condition}. Finally, $\theta^*$ is unique up to an additive constant, since $A^\top \mathbbold{1}=\mathbbold{0}$ and the graph is connected~\cite[Theorem~5.1]{n_torus}.
\end{proof}

We are now interested in deriving a sufficient condition for a given set of power injections $P$ and network topology, such that the line flow $f$ obtained satisfies conditions~\ref{item:cond1} and~\ref{item:cond2} of Proposition~\ref{power_flow_solution}. These conditions guarantee the existence of a unique power flow solution for the specified nodal power injection up to an additive constant.

\section{Sufficient Condition for Power Flow Equation Solvability}
\label{Suff_Condition}

To derive a sufficient condition for power flow solvability, we seek a family of line flow solutions that preserves the nodal balance equation {\eqref{nodal_balance}} while allowing adjustments in the cycle flow. By \eqref{feasible_solution}, every line flow solution can be written as
\begin{equation}
f(\lambda)=\hat f + C_{\Sigma}\lambda, \qquad \lambda \in \mathbb{R}^q,
\label{eq:cycle_parameterization}
\end{equation}
where $q=m-n+1$ and $C_{\Sigma}\in\mathbb{R}^{m\times q}$ is the cycle basis matrix. That is, $f(\lambda)$ is a line flow solution parameterized by cycle flow.

We next define the line flow vector normalized by $D$ as
\begin{equation}
z(\lambda) \coloneqq D^{-1}f(\lambda) = D^{-1}\hat f + D^{-1}C_{\Sigma}\lambda = z_0 + H\lambda,
\label{eq:z_lambda}
\end{equation}
where $z_0 \coloneqq D^{-1}\hat f$ and $H \coloneqq D^{-1}C_{\Sigma}$. We note the edge-wise feasibility condition in Proposition~\ref{power_flow_solution} requires
\begin{equation}
|z_{ij}(\lambda)| \le 1, \quad (i,j)\in\E.
\label{eq:edge_feas_box}
\end{equation}
{Thus, the feasible domain can be defined as
\begin{equation}
    \Omega
    =
    \left\{
    \lambda\in\mathbb{R}^q :
    |z_{ij}(\lambda)|\le1,\ 
    \forall (i,j)\in\mathcal{E}
    \right\}.
    \label{eq:feasible_domain}
\end{equation}
Inside the feasible domain $\Omega$, the argument of the inverse sine function remains admissible for all edges.}
Now we can denote the cycle consistency expression in \eqref{cycle_consistent} parametrized by the cycle flow indicator $\lambda$ as
\begin{equation} \label{eq:g_lambda}
    g(\lambda) \coloneqq C_{\Sigma}^{\top}\arcsin(z(\lambda)) = C_{\Sigma}^{\top}\arcsin(z_0 + H\lambda),
\end{equation}
where $g(\lambda)$ is a $q$-dimensional map whose components represent the sum of phase angle differences around the corresponding basis cycles. Therefore, by Proposition~\ref{power_flow_solution}, a {necessary and} sufficient condition for the existence of a power flow solution {satisfying \eqref{main_condition}} is the existence of {a $\lambda \in \Omega$} that satisfies $g(\lambda) = \mathbbold{0}$. {A key observation, upon which our solvability condition is built, is that $g(\lambda)$ is a monotone function. This fact is substantiated in the following lemma:}
\begin{lemma} \label{thm:monotone}
    {The function $g(\lambda)$ defined in \eqref{eq:g_lambda} is strongly monotone for $\lambda \in \inte \Omega$. In particular,
    \begin{equation}
        \nabla g(\lambda) \succeq M := C_\Sigma^\top D^{-1}C_\Sigma, \quad \lambda\in\inte\Omega.
    \end{equation}
    Furthermore, it is the gradient of a strongly convex function.}
\end{lemma}
\begin{proof}
    {The Jacobian of $g(\lambda)$ is
    \begin{equation} \label{eq:gradg}
        \nabla g(\lambda) = C_\Sigma^\top W(\lambda) H = C_\Sigma^\top W(\lambda) D^{-1}C_\Sigma
    \end{equation}
    where $W(\lambda)$ is a diagonal matrix with diagonal elements $(1-z_{ij}(\lambda)^2)^{-\frac{1}{2}}, (i,j)\in\mathcal{E}$. When $\lambda \in \inte \Omega$, $W(\lambda) \succeq I$ holds. In addition, $D^{-1}$ is also a positive definite matrix given positive voltages and inductive lines. Furthermore, $C_\Sigma$ has full column rank since the cycle bases are linearly independent. It follows that $\nabla g(\lambda) \succeq M \succ 0$, which means $g(\lambda)$ is strongly monotone for $\lambda \in \inte \Omega$. Since $\nabla g(\lambda)$ is symmetric for all $\lambda\in\inte\Omega$, $g$ admits a scalar potential $\Phi$ on $\inte\Omega$. Furthermore, $\Phi$ is strongly convex on $\inte \Omega$ since $\nabla g \succeq M$.}
\end{proof}

{Let $\lambda_0 \in \inte\Omega$ be a given coordinate vector. We next present an explicit sufficient condition that certifies the existence of a perturbation $\Delta\lambda$ such that $\lambda_0+\Delta\lambda \in \Omega$ and $g(\lambda_0+\Delta\lambda)=\mathbbold{0}$. The condition is given in terms of the weighted Euclidean norms induced by $M \succ 0$ defined as follows: For a positive definite matrix $M\succ 0$, the $M$-norm of the vector $x$ is defined as $\|x \|_M \coloneqq \sqrt{x^\top M x}$; and its dual norm is defined as $\|x \|_{M,*} \coloneqq \sqrt{x^\top M^{-1} x}$.}

\begin{theorem}[Sufficient Power Flow Solvability Condition]
\label{thm:sufficient_condition} Consider a connected power system with $P \in \mathbbold{1}^\perp$. Let $f(\lambda)$, $z(\lambda)$, and $g(\lambda)$ be defined as in \eqref{eq:cycle_parameterization}, \eqref{eq:z_lambda}, and \eqref{eq:g_lambda}, respectively. Consider a point $\lambda_0 \in \inte\Omega$. Define $r$ as
\begin{equation}
    r \coloneqq
    \min_{(i,j) \in \E, \| H_{ij,:}\|_{M,*} > 0}
    \frac{1-|z_{ij}(\lambda_0)|}{\|H_{ij,:}\|_{M,*}},
    \label{Euclidian_Ball}
\end{equation}
where $M \coloneqq C_\Sigma^\top D^{-1}C_\Sigma$. If 
\begin{equation} \label{eq:condition}
    \boxed{r > \| g(\lambda_0)\|_{M,*}}, 
\end{equation}
then there exists a unique $\lambda^* \in \inte\Omega$ such that $g(\lambda^*) = \mathbbold{0}$ with $\|\lambda^* - \lambda_0\|_M \le \|g(\lambda_0)\|_{M,*}$. Consequently, $f(\lambda^*)$ is a line flow solution satisfying conditions \ref{item:cond1} and \ref{item:cond2} in Proposition \ref{power_flow_solution}, and the real power flow equations \eqref{PF:matrix_form} admit a unique solution up to an additive constant satisfying \eqref{main_condition}.
\end{theorem}
\vspace{1mm}
\begin{proof}
{First, we characterize the conditions on the existence of $\Delta\lambda$ such that $\lambda_0 + \Delta\lambda \in \Omega$ and $g(\lambda_0 + \Delta\lambda) = \mathbbold{0}$. We first derive the condition under which $\lambda_0 + \Delta\lambda \in \Omega$. For this to hold, condition~\eqref{eq:edge_feas_box} needs to be satisfied, i.e.,
\begin{equation}
    |z_{ij}(\lambda_0+\Delta\lambda)| \le 1,
     \quad (i,j)\in\E .
\end{equation}
Since  $z(\lambda_0+\Delta\lambda) =z(\lambda_0)+H\Delta\lambda$, we require
\begin{equation}
    |z_{ij}(\lambda_0)+H_{ij,:}\Delta\lambda| \le 1, \quad (i,j)\in\E.
\end{equation}
Using the triangle inequality and Cauchy-Schwarz inequality in the $M$-norm and its dual norm, we have
\begin{multline}
    |z_{ij}(\lambda_0)+H_{ij,:}\Delta\lambda| = |z_{ij}(\lambda_0)+H_{ij,:}M^{-1} M \Delta\lambda| \\
    \leq
    |z_{ij}(\lambda_0)| +
    \|H_{ij,:}\|_{M,*}\|\Delta\lambda\|_M, \quad (i,j)\in\E. 
\end{multline}
Hence, feasibility is guaranteed whenever $\|\Delta\lambda\|_M \le (1-|z_{ij}(\lambda_0)|)/\|H_{ij,:}\|_{M,*}$ for all edges with $\|H_{ij,:}\|_{M,*}>0$. In other words, a sufficient condition for $\lambda_0 + \Delta\lambda \in \Omega$ is}
\begin{equation}
    {\|\Delta\lambda\|_M
    \le
    \min_{(i,j) \in \E, \| H_{ij,:}\|_{M,*} > 0}
    \frac{1-|z_{ij}(\lambda_0)|}{\|H_{ij,:}\|_{M,*}} = r.}
\end{equation}

{Let $d\coloneqq \|g(\lambda_0)\|_{M,*}$. If $d = 0$, then $g(\lambda_0) = \mathbbold{0}$ and the result follows by taking $\lambda^* = \lambda_0$. Suppose $d > 0$. Since $d < r$ by condition \eqref{eq:condition}, we choose $\rho$ such that $d < \rho < r$. Consider the $M$-ball
\begin{equation}
    \mathcal{B}_\rho \coloneqq \{ \lambda_0 + \Delta\lambda: \|\Delta\lambda\|_M \le \rho \}.
\end{equation}
The preceding analysis has shown that $\mathcal{B}_\rho \subset \inte\Omega$.}

{According to Lemma \ref{thm:monotone}, $\nabla g(\lambda) \succeq M$ for all $\lambda\in\inte\Omega$. Since $\mathcal{B}_\rho \subset \inte\Omega$, for any $\Delta\lambda$ with $\|\Delta\lambda\|_M \le \rho$, we have
\begin{multline}
    \hspace{-0.1in}\bigl( g(\lambda_0+\Delta\lambda) - g(\lambda_0) \bigr)^\top \Delta \lambda = \int_0^1 \Delta\lambda^\top \nabla g(\lambda_0 + t\Delta\lambda)\Delta\lambda \,\mathrm{d}t \\
    \ge \Delta\lambda^\top M \Delta\lambda = \|\Delta\lambda\|_M^2.
\end{multline}
It follows
\begin{align}
    \hspace{-0.05in}g^\top(\lambda_0+\Delta\lambda) \Delta\lambda &\ge \|\Delta\lambda\|_M^2 + g^\top(\lambda_0) \Delta\lambda \nonumber \\
    &\ge \|\Delta\lambda\|_M^2 - \|g(\lambda_0)\|_{M,*}\|\Delta\lambda\|_M \nonumber\\
    &= \|\Delta\lambda\|_M(\|\Delta\lambda\|_M - \|g(\lambda_0)\|_{M,*}). \label{eq:monotone:2}
\end{align}
Therefore, for perturbations satisfying $\|\Delta\lambda\|_M=\rho$, we obtain $g^\top(\lambda_0+\Delta\lambda) \Delta\lambda \ge \rho(\rho-d) > 0$.}

{By Lemma \ref{thm:monotone}, $g(\lambda)$ is the gradient of a twice-differentiable strongly convex function $\Phi(\lambda)$ on $\inte\Omega$. Since $B_\rho$ is compact and
$B_\rho\subset\operatorname{int}\Omega$, $\Phi$ attains a unique
minimizer on $B_\rho$ \cite[Sect. 4.2.1]{boyd2004convex}. By the optimality condition for differentiable convex minimization over a convex set \cite[Sect. 4.2.3]{boyd2004convex}, any minimizer $\lambda^*$ satisfies
\begin{equation}
    \nabla \Phi^\top(\lambda^*)(\lambda^* - \lambda_0) = g^\top(\lambda^*) (\lambda^* - \lambda_0) \le 0.
\end{equation}
This contradicts condition \eqref{eq:monotone:2} that $g^\top (\lambda^*) (\lambda^* - \lambda_0) > 0$ for any $\lambda^*$ on the boundary of $\mathcal{B}_\rho$. Therefore, the unique minimizer must lie in the interior, and the unconstrained optimality condition is}
\begin{equation}
    \nabla\Phi(\lambda^*) = g(\lambda^*) =\mathbbold{0}.
\end{equation}

{It remains to show the stated bound. Since $g(\lambda^*)=\mathbbold{0}$, strong monotonicity gives
\begin{align}
    \|\lambda^*-\lambda_0\|_M^2
    &\le
    \bigl(g(\lambda^*)-g(\lambda_0)\bigr)^\top
    (\lambda^*-\lambda_0) \notag\\
    &=
    -g(\lambda_0)^\top(\lambda^*-\lambda_0) \notag\\
    &\le
    \|g(\lambda_0)\|_{M,*}\|\lambda^*-\lambda_0\|_M .
\end{align}
Therefore,}
\begin{equation}
    {\|\lambda^*-\lambda_0\|_M \le \|g(\lambda_0)\|_{M,*}.}
\end{equation}

{Finally, since \(\lambda^*\in\operatorname{int}\Omega\),
the line flow \(f(\lambda^*)\) satisfies the edge-wise feasibility
condition. Moreover, \(g(\lambda^*)=\mathbbold{0}\) is exactly the cycle
consistency condition. Thus, by Proposition \ref{power_flow_solution}, $f(\lambda^*)$
corresponds to a unique solution of the real power flow equations satisfying \eqref{main_condition} up to an additive constant.}
\end{proof}

\begin{remark}
    {The sufficient condition in Theorem \ref{thm:sufficient_condition} is invariant under the change of cycle basis as long as the initial $\lambda_0$ is scaled accordingly. In particular, it is invariant when the initial $\lambda_0$ is obtained via one Newton step from zero.}

    {Let $\Sigma'$ be an alternative cycle basis whose corresponding cycle basis matrix is $C_{\Sigma'}$. There exists a full-rank matrix $T \in \mathbb{R}^{q\times q}$ such that $C_{\Sigma'} = C_{\Sigma}T$. For the original cycle basis, the initial cycle flow coordinate via one-step Newton is $\lambda_0 = -J_0^{-1} g(0)$ where $J_0 = C_\Sigma^\top W(0) D^{-1}C_\Sigma$. For the alternative cycle basis, the initial cycle flow coordinate $\mu_0$ is given by $\mu_0 = -(J_0')^{-1}g'(0)$ with $J_0' = T^\top C_\Sigma^\top W(0) D^{-1}C_\Sigma T = T^\top J_0 T$ and $g'(0) = T^\top g(0)$. It follows $\lambda_0 = T\mu_0$. Consequently, the line flows under the initial cycle flow vectors are the same, as $z_0 + D^{-1}C_\Sigma T\mu_0 = z_0 + D^{-1}C_\Sigma \lambda_0$.}

    {In addition, the cycle consistency vectors under the two cycle bases are related by $g'(\lambda) = T^\top g(T\lambda)$. Similarly, $M' \coloneqq C_{\Sigma'}^\top D^{-1} C_{\Sigma'} = T^\top M T$ and $H' \coloneqq D^{-1}C_{\Sigma'} = HT$. It is then easy to verify that $\| g(\lambda_0)\|_{M,*}$ in condition \eqref{eq:condition} stays the same with a change of cycle basis, the same holds true for the radius $r$ in the same equation.} 
    
    {We have thus shown that condition \eqref{eq:condition} is invariant under the change of cycle basis. This suggests that the condition reflects the fundamental properties of the weighted cycle space rather than a particular coordinate system that describes it.}
\end{remark}

\section{Cycle Basis Characterization and Solvability Verification}
\label{Methodology}

In this section, we describe how the cycle basis is constructed and how the sufficient condition for solvability derived in Section \ref{Suff_Condition} can be verified. The proposed procedure is carried out entirely in the line flow domain. We start from the particular line flow solution $\hat f = DA^\top L^{\dagger} P$, construct a cycle basis matrix $C_\Sigma$, form the cycle flow parametrization by $\lambda \in \mathbb{R}^{q}$ as $f(\lambda) = \hat f + C_\Sigma \lambda$, and verify the theorem.

\subsection{Cycle Basis Construction}

In our implementation, the cycle basis is obtained through a Depth-First Search (DFS) algorithm. In this algorithm, an arbitrary node is first selected as the root of the DFS tree. Whenever DFS explores an edge from a visited node to an unvisited node, that edge is classified as a tree edge and oriented from the child toward the parent. This produces a directed spanning tree. Any non-tree edge encountered during DFS, together with the unique directed tree path connecting its end nodes, forms a directed fundamental cycle. Each such cycle contributes one column to the cycle basis matrix $C_\Sigma$. By changing the starting node of the DFS algorithm, different cycle bases can be obtained.

\subsection{Verification Procedure}

After discussing the cycle basis construction, we describe how to apply the sufficient solvability condition~{\eqref{eq:condition}} to certify the existence {and uniqueness} of a power flow solution.

With the cycle basis matrix $C_\Sigma$, we compute the particular line flow $\hat f = DA^\top L^{\dagger} P$, normalized line flow $z_{0} = D^{-1} \hat f$, and the normalized cycle basis matrix $H = D^{-1} C_\Sigma$. The parametrized line flows and the cycle consistency residual can subsequently be defined as $f(\lambda) = \hat f + C_\Sigma \lambda$, $z(\lambda) = z_{0} + H \lambda$, and $g(\lambda) = C_\Sigma^{\top} \arcsin(z(\lambda))$.

After the definitions above, the verification procedure can be broken down into the following two steps. First, one iteration of NR is used to get $\lambda_0 \in \inte\Omega$. If $\lambda_0 \notin \inte\Omega$, the certificate is inconclusive. Otherwise, the feasible $M$-ball of radius $r$ centered at $\lambda_0$ is then constructed inside the feasible domain $\Omega$. Second, we compute the cycle residual $\| g(\lambda_0)\|_{{M,*}}$.

Now, if $r>\| g(\lambda_0)\|_{{M,*}}$, it follows that there exists a perturbation $\Delta\lambda$ inside the feasible domain such that $g(\lambda_0+\Delta\lambda)=\mathbbold{0}.$
Hence, there exists $\lambda^*=\lambda_0+\Delta\lambda$ such that $g(\lambda^*)=\mathbbold{0}$. Therefore, the line flow $f(\lambda^*)=\hat f+C_\Sigma\lambda^*$ satisfies both edge-wise feasibility and cycle consistency conditions.  By Proposition~\ref{power_flow_solution}, it therefore corresponds to a solution of the lossless real power flow equation.

Algorithm~\ref{alg:feasibility_test} summarizes the verification procedure. The DFS traversal has computational complexity
\(O(|\mathcal N|+|\mathcal E|)\) since each node and edge is visited at most
once. Constructing the explicit cycle basis matrix \(C_\Sigma\) requires
additional work proportional to the total length of the fundamental cycles. The computationally expensive step in the proposed algorithm is the inversion of the matrix $M = C_\Sigma^\top D^{-1} C_\Sigma$, which is required for the computation of $r$ and $\| g(\lambda_0)\|_{{M,*}}$. However, this matrix depends only on the network topology, line parameters, and fixed voltage magnitudes. Therefore, it needs to be computed only once for a given network topology. All remaining steps involve simple algebraic operations. Although one Newton step is used to obtain a feasible initial point, it is relatively inexpensive compared to repeatedly solving the full nonlinear power flow equations.

\begin{algorithm}[t]
\caption{{Sufficient Solvability Condition Verification}}
\label{alg:feasibility_test}
\begin{algorithmic}[1]
\STATE Construct the cycle basis matrix $C_\Sigma$ and node-edge incidence matrix $A$, $M =C_\Sigma^\top D^{-1}C_\Sigma$
\STATE Compute $\hat f = DA^\top L^{\dagger}P$
\STATE Compute $z_0=D^{-1}\hat f$ and $H=D^{-1}C_\Sigma$
\STATE Find a feasible initial point $\lambda_0$ satisfying $|z(\lambda_0)|<1$
\STATE Compute $g(\lambda_0)=C_\Sigma^\top\arcsin(z_0+H\lambda_0)$
\STATE Compute $\| g(\lambda_0)\|_{{M,*}}$
\STATE Compute
\begin{equation*}  
r=    \min_{(i,j) \in \E, \| H_{ij,:}\|_{M,*} > 0}
    \frac{1-|z_{ij}(\lambda_0)|}{\|H_{ij,:}\|_{M,*}},
\end{equation*}
\IF{$r>\| g(\lambda_0)\|_{{M,*}}$}
\STATE Solvability certified by Theorem~\ref{thm:sufficient_condition}
\ELSE
\STATE Inconclusive
\ENDIF
\end{algorithmic}
\end{algorithm}

\section{Results}
\label{Results}
To assess the accuracy of the proposed sufficient condition, we evaluated its performance on different IEEE networks and compared the {results with the empirical loading limits obtained using the Continuation Power Flow (CPF) method.} The comparison is based on the maximum system stress level for which a feasible power flow solution exists. For larger networks, NR often underestimated the solvability limit. Therefore, CPF was used as the benchmark, since it provides a more reliable estimate of the actual solvability boundary.

System stress is introduced by uniformly scaling the power injections by a scalar parameter $y$, i.e., $P(y)=y P^0$, where $P^0$ denotes the nominal injection vector. The corresponding particular line flow solution is
\begin{equation}
\hat f(y)=DA^\top L^\dagger P(y)
= y\,DA^\top L^\dagger P^0.
\end{equation}

For each value of $y$, we apply the verification procedure described in Section~\ref{Methodology}. Let $y_{\mathrm{cert}}$ denote the maximum loading level for which the proposed sufficient condition certifies the existence of a power flow solution. Likewise, let $y_{\mathrm{CPF}}$ denote the largest loading level for which the CPF method converges to a solution of the lossless power flow equations when initialized in the standard manner, and the phase angle differences are bounded by $\pi/2$. The quantity $y_{\mathrm{cert}}$ therefore represents the certified solvability margin given by the proposed method, whereas $y_{\mathrm{CPF}}$ serves as a reference for the actual solvability boundary. To quantify the tightness of the certificate, we report $\eta \coloneqq \frac{y_{\mathrm{cert}}}{y_{\mathrm{CPF}}}$ as the tightness measure. A value of $\eta$ closer to $1$ indicates a less conservative certificate. The CPF simulations were carried out using \textsc{Matpower}~\cite{Matpower_Data}.

Table~\ref{tab:critical_ratio} compares the proposed test with two existing sufficient conditions from the literature, namely the $\lambda_2$ test~\cite{Lambda_2_test} and the $\infty$-norm test~\cite{Cutset_Projection}. It can be seen from the table that {for different test networks}, the proposed method yields the best certificate among all existing conditions, with the certified loading level close to the CPF-based reference. These results show that the proposed cycle constraint approach can provide a remarkably tight sufficient solvability condition for meshed networks of different sizes.

{To examine the impact of cycle basis selection on the proposed certification, we randomly varied the starting node used in the DFS-based cycle basis construction and repeated the process several times. The accuracy results obtained across all random samples remained identical, which supports the basis-invariance property of the proposed certificate.}

\begin{table}[t]
\centering
\caption{Certified stress ratio $y_{\mathrm{cert}}/y_{\mathrm{CPF}}$ for different test cases.}
\label{tab:critical_ratio}
\renewcommand{\arraystretch}{1.3}
\rowcolors{2}{gray!25}{white}
\begin{tabular}{lccc}
\toprule
\textbf{Test Case} 
& $\lambda_2$ test~\cite{Lambda_2_test}
& $\infty$-norm test~\cite{Cutset_Projection}
& Proposed test \\
\midrule
IEEE 9      & 16.01\% & 74.16\% & 98.71\% \\
IEEE 14     & 6.84\%  & 48.60\% & 83.66\% \\
IEEE RTS 24 & 3.33\%  & 51.55\% & 99.84\% \\
IEEE 30     & 2.41\%  & 48.56\% & 85.76\% \\
IEEE 39     & 3.24\%  & 67.43\% & 100\% \\
IEEE 57     & 0.61\%  & 56.53\% & 95.89\% \\
IEEE 118    & 0.20\%  & 41.04\% & 90.74\% \\
IEEE 300    & 0.00\%  & 39.67\% & 98.83\% \\
Polish 2383 & 0.02\%  & 23.05\% & 77.04\% \\
\bottomrule
\end{tabular}
\end{table}

\section{Conclusion}
\label{Conclusion}

In this paper, we developed a new sufficient condition for the solvability of the lossless real power flow equation based on the cycle constraints. {By developing a precise characterization of realizable line flows and exploiting the strong monotonicity of a cycle flow-based reformulation of the real power flow equations}, we obtained a mathematically rigorous and practically accurate condition for power flow solvability. Numerical results on different IEEE networks show that the proposed condition can be significantly less conservative than existing sufficient conditions and can provide a remarkably tight certificate of solvability.

The proposed framework provides a cycle-aware alternative to conventional DC-based approximations in settings where a rigorous sufficient condition for AC power flow solvability is desired. This work builds upon the DC power flow, which is always feasible, by preserving the nonlinear trigonometric terms to better approximate the AC power flow equations. The sufficient condition can be incorporated in optimal power flow, security assessment, and other applications where tractable yet physically meaningful feasibility guarantees are needed.

 
\bibliographystyle{IEEEtran}
\bibliography{Ref}

\end{document}